\documentclass[times, doublespace]{oupau}

\usepackage{enumerate}
\usepackage{amsmath, amssymb, amscd, mathrsfs}
\usepackage{xypic}

\newcommand{\CSTAR}{C$^*$-}
\DeclareMathOperator{\Kconv}{\bbK_{conv}}

\DeclareMathOperator{\States}{States}

\DeclareMathOperator{\Pairing}{Pairing}

\newtheorem{thm}{Theorem}[section]
\newtheorem{conj}[thm]{Conjecture}
\newtheorem{claim}[theorem]{Claim}
\newtheorem{prop}[thm]{Proposition}

\newtheorem{remark}[theorem]{Remark}

\numberwithin{equation}{section}

\def\N{{\mathbb N}}

\def\Q{{\mathbb Q}}

\newcommand{\bbF}{{\mathbb F}}

\newcommand{\bbK}{{\mathbb K}}
\newcommand{\bbS}{{\mathbb S}}

\newcommand{\bbT}{{\mathbb T}}

\DeclareMathOperator{\Cu}{Cu}
\newcommand{\BCu}{\mathbf C}
\newcommand{\lCu}{\lesssim}

\newcommand{\sCu}{\sim}

\newcommand{\p}{\mathfrak p}

\newcommand{\cP}{{\mathcal P}}

\newcommand{\bfS}{{\mathbf S}}
\newcommand{\bfG}{{\mathbf G}}
\newcommand{\bfGord}{{\mathbf G_{\text{ord}}}}
\newcommand{\bfGordu}{{\mathbf G'_{\text{ord}}}}

\newcommand{\bbN}{{\mathbb N}}
\newcommand{\bbC}{\mathbb C}

\newcommand{\bbR}{\mathbb R}

\newcommand{\cX}{{\mathcal X}}

\newcommand{\cY}{{\mathcal Y}}
\newcommand{\cZ}{{\mathcal Z}}

\newcommand{\rs}{\restriction}

\DeclareMathOperator{\Proj}{Proj}

\DeclareMathOperator{\Sym}{Sym}
\DeclareMathOperator{\im}{im}
\DeclareMathOperator{\Tensor}{Tensor}

\DeclareMathOperator{\Kzero}{K_0}
\DeclareMathOperator{\Kzerou}{K_{0,u}}
\DeclareMathOperator{\Kone}{K_1}

\newcommand{\cC}{\mathcal C}
\newcommand{\cF}{\mathcal F}

\newcommand{\cB}{\mathcal B}
\newcommand{\cK}{\mathcal K}

\newcommand{\e}{\varepsilon}

\newcommand{\bfEll}{{\mathbf{Ell}}}
\newcommand{\Ell}{Ell}

\newcommand{\DeltaN}{\Delta^{\bbN}}

\DeclareMathOperator{\Au}{\mathfrak A_u}
\DeclareMathOperator{\Gammau}{\Gamma_u}

\DeclareMathOperator{\dimnuc}{\mathrm{dim}_{nuc}}
\DeclareMathOperator{\Th}{\mathrm{Th}}

\title{The descriptive set theory of \CSTAR algebra invariants}

\author{Ilijas Farah\affil{1}, Andrew Toms\affil{2} and Asger T\"ornquist\affil{3}\\ (Appendix with Caleb Eckhardt)}

\abbrevauthor{I. Farah, A. Toms and A. T\"ornquist}

\headabbrevauthor{Farah, I., Toms, A., and T\"ornquist, A.}

\address{%
\affilnum{1} Department of Mathematics and Statistics, York University, 4700 Keele Street, North York, Ontario, Canada, M3J 1P3, and Matematicki Institut, Kneza Mihaila 34, Belgrade, Serbia (ifarah@mathstat.yorku.ca)
and
\affilnum{2} Department of Mathematics, Purdue University, 150 N. University St., West Lafayette, IN 47902, USA (atoms@purdue.edu)
and
\affilnum{3} Department of Mathematics, University of Copenhagen, Universitetsparken 5, 2100 Copenhagen, Denmark (asgert@math.ku.dk)}

\correspdetails{asgert@math.ku.dk}

\begin{document}

\begin{abstract}
We establish the Borel computability of various \CSTAR algebra invariants, including the Elliott invariant and the Cuntz semigroup.  As applications we deduce that AF algebras are classifiable by countable structures, and that a conjecture of Winter and the second author for nuclear separable simple \CSTAR algebras cannot be disproved by appealing to known standard Borel structures on these algebras.
\end{abstract}

\received{December 25, 2011}

\maketitle

\section{Introduction}\label{S.intro}

The classification theory of nuclear separable \CSTAR algebras via $\mathrm{K}$-theoretic and tracial invariants was initiated by G. A. Elliott c. 1990.  An ideal result in this theory is of the following type:

\vspace{2mm}
\begin{quote}
Let $\cC_1$ be a category of \CSTAR algebras, $\cC_2$ a category of invariants, and $\cF:\cC_1 \to \cC_2$ a functor.  We say that $(\cF, \cC_2)$ {\it classifies} $\cC_1$ if for any isomorphism $\phi:\cF(A) \to \cF(B)$ there is an isomorphism $\Phi:A \to B$ such that $\cF(\Phi) =\phi$, and if, moreover, the range of $\cF$ can be identified.
\end{quote}

\vspace{2mm}
\noindent
Given $A,B \in \cC_1$, one wants to decide whether $A$ and $B$ are isomorphic.  With a theorem as above in hand (and there are plenty such---see \cite{et} or \cite{Ror:Classification} for an overview), this reduces to deciding whether $\cF(A)$ and $\cF(B)$ are isomorphic;  in particular, one must compute $\cF(A)$ and $\cF(B)$.  What does it mean for an invariant to be computable?  The broadest definition is available when the objects of $\cC_1$ and $\cC_2$ admit natural parameterizations as standard Borel spaces, for the computability of $\cF(\bullet)$ then reduces to the question ``Is $\cF$ a Borel map?"  The aim of this paper is to prove that a variety of \CSTAR algebra invariants are indeed Borel computable, and to give some applications of these results.

Our main results are summarized informally below.

\begin{theorem}\label{invariants}
The following invariants of a separable \CSTAR algebra $A$ are Borel computable:
\begin{itemize}
\item[(i)] the (unital) Elliott invariant $\mathrm{Ell}(A)$ consisting of pre-ordered $\mathrm{K}$-theory, tracial functionals, and the pairing between them;
\item[(ii)] the Cuntz semigroup $\mathsf{Cu}(A)$;
\item[(iii)] the radius of comparison of $A$;
\item[(iv)] the real and stable rank of $A$;
\item[(v)] the nuclear dimension of $A$;
\item[(vi)] the presence of $\mathcal{Z}$-stability for $A$;
\item[(vii)] the theory $\mathrm{Th}(A)$ of $A$.
\end{itemize}
\end{theorem}

\noindent
Proving that the Elliott invariant and the Cuntz semigroup are computable turn out to be the most involved tasks.  (The latter in particular.)  

From a descriptive set theoretic point of view, a classification problem is a pair $(X,E)$ consisting of a standard Borel space $X$, the (parameters for) objects to be classified, and an equivalence relation $E$, the relation of isomorphism among the objects in $X$. In most interesting cases, the equivalence relation $E$ is easily definable from the elements of $X$ and is seen to be Borel or, at worst, analytic; that is certainly the case here.  To compare the relative difficulty of classification problems $(X,E)$ and $(Y,F)$, we employ the notion of Borel reducibility: One says that $E$ is Borel reducible to $F$ if there is a Borel map $\Theta:X \to Y$ with the property that
\[
xEy \iff \Theta(x)F\Theta(y).
\]
The relation $F$ is viewed as being at least as complicated as $E$.  The relation $E$ is viewed as being particularly nice when $F$-classes are ``classifiable by countable structures".  Equivalently (\cite{frst89}), the relation $E$ is no more complicated than isomorphism for countable graphs.  Theorem \ref{invariants} (i) entails the computability of the pointed (pre-)ordered $\mathrm{K}_0$-group of a unital separable \CSTAR algebra.  As isomorphism of such groups is Borel-reducible to isomorphism of countable graphs, we have the following result.  

\begin{theorem}\label{AFccs}
AF algebras are classifiable by countable structures.
\end{theorem}

In order to classify nuclear separable \CSTAR algebras using only $\mathrm{K}$-theoretic and tracial invariants, it is necessary to assume that the algebras satisfy some sort of regularity property, be it 
topological,  homological or  \CSTAR algebraic  (see \cite{et} for a survey).  This idea is summarized in the following conjecture of Winter and the second author.

\begin{conj}\label{toms-winter}
Let $A$ be a simple unital separable nuclear and infinite-dimensional \CSTAR algebra.  The following are equivalent:
\begin{enumerate}
\item[(i)] $A$ has finite nuclear dimension;
\item[(ii)] $A$ is $\mathcal{Z}$-stable;
\item[(iii)] $A$ has strict comparison of positive elements.
\end{enumerate}
\end{conj}

\noindent
Combining the main result of \cite{W:nucz} with that of \cite{Rob:ncomp} yields (i) $\Rightarrow$ (ii), while R\o rdam proves (ii) $\Rightarrow$ (iii) in \cite{Ro:zsr}.  Partial converses to these results follow from the successes of Elliott's classification program.  Here we prove the following result.

\begin{theorem}
The classes (i), (ii), and (iii) of Conjecture \ref{toms-winter} form Borel sets.
\end{theorem}

\noindent
Therefore the classes of \CSTAR algebras appearing in the conjecture have the same 
 descriptive set theoretic complexity.  

The article is organized as follows:  in Section \ref{S.preliminaries} we recall two parameterizations of separable \CSTAR algebras as standard Borel spaces;  Section \ref{S.Elliott} establishes the computability of the Elliott invariant;  
Section \ref{S.Cuntz} considers the computability of the Cuntz semigroup and the radius of comparison;  Section \ref{S.other} and the Appendix deal with $\mathcal{Z}$-stability, nuclear dimension, the first-order theory of a \CSTAR algebra in the logic of metric structures, and the real and stable rank.

\section{Preliminaries}\label{S.preliminaries}
 In \cite{FaToTo2} we introduced four parameterizations of separable \CSTAR algebras by standard Borel spaces and proved that they were equivalent.  Here we'll need only two, which we recall for the reader's convenience.  
 

\subsection{The space $\Gamma(H)$.} \label{ss.gamma} 
Let $H$ be a separable infinite dimensional Hilbert space and let as usual $\cB(H)$ denote the space of bounded operators on $H$. The space $\cB(H)$ becomes a standard Borel space when equipped with the Borel structure generated by the weakly open subsets. Following \cite{Kec:C*} we let
$$
\Gamma(H)=\cB(H)^{\bbN},
$$
and equip this with the product Borel structure. For each $\gamma\in\Gamma(H)$ we let $C^*(\gamma)$ be the \CSTAR algebra generated by the sequence $\gamma$. If we identify each $\gamma\in\Gamma(H)$ with $C^*(\gamma)$, then naturally $\Gamma(H)$ parameterizes all separable \CSTAR algebras acting on $H$. Since every separable \CSTAR algebra is isomorphic to a \CSTAR subalgebra of $\cB(H)$ this gives us a standard Borel parameterization of the category of all separable \CSTAR algebras. If the Hilbert space $H$ is clear from the context we will write $\Gamma$ instead of $\Gamma(H)$. We define 
$$
\gamma\simeq^{\Gamma}\gamma'\iff C^*(\gamma)\text{ is isomorphic to } C^*(\gamma').
$$


\subsection{The space $\hat\Gamma(H)$.} \label{ss.hatgamma}
Let $\Q(i)=\Q+i\Q$ denote the complex rationals. Following \cite{Kec:C*}, let $(\p_j: j\in \N)$ enumerate the non-commutative $*$-polynomials without constant term in the formal variables $X_k$, $k\in\N$, with coefficients in $\Q(i)$, and for $\gamma\in \Gamma$ write $\p_j(\gamma)$ for the evaluation of $\p_j$ with $X_k=\gamma(k)$. Then $C^*(\gamma)$ is the norm-closure of $\{\p_j(\gamma): j\in \bbN\}$. The map $\Gamma\to\Gamma:\gamma\mapsto\hat\gamma$ where $\hat\gamma(j)=\p_j(\gamma)$ is clearly a Borel map from $\Gamma$ to $\Gamma$. If we let
$$
\hat\Gamma(H)=\{\hat\gamma:\gamma\in\Gamma(H)\},
$$
then $\hat\Gamma(H)$ is a standard Borel space and provides another parameterization of the \CSTAR algebras acting on $H$; we suppress $H$ and write $\hat\Gamma$ whenever possible. For $\gamma\in\hat\Gamma$, let $\check\gamma\in\Gamma$ be defined by 
$$
\check\gamma(n)=\gamma(i)\iff \p_i=X_n,
$$
and note that $\hat\Gamma\to\Gamma:\gamma\mapsto\check\gamma$ is the inverse of $\Gamma\to\hat\Gamma:\gamma\mapsto\hat\gamma$. We let $\simeq^{\hat\Gamma}$ be
$$
\gamma\simeq^{\Gamma}\gamma'\iff C^*(\gamma)\text{ is isomorphic to } C^*(\gamma').
$$
It is clear from the above that $\Gamma$ and $\hat\Gamma$ are equivalent parameterizations.

An alternative picture of $\hat\Gamma(H)$ is obtained by considering the free (i.e., surjectively universal) countable unnormed $\Q(i)$-$*$-algebra $\mathfrak A$. We can identify $\mathfrak A$ with the set $\{\p_n:n\in\N\}$. Then
$$
\hat \Gamma_{\mathfrak A}(H)=\{f:\mathfrak A\to\mathcal B(H):f\text{ is a $*$-homomorphism}\}
$$
is easily seen to be a Borel subset of $\cB(H)^{\mathfrak A}$. For $f\in\hat\Gamma_{\mathfrak A}$ let $C^*(f)$ be the norm closure of $\im(f)$, and define
$$
f\simeq^{\hat\Gamma_{\mathfrak A}} f'\iff   C^*(f)\text{ is isomorphic to } C^*(f').
$$
Clearly the map $\hat\Gamma\to\hat\Gamma_{\mathfrak A}:\gamma\mapsto f_\gamma$ defined by $f_\gamma(\p_j)=\gamma(j)$ provides a Borel bijection witnessing that $\hat\Gamma$ and $\hat\Gamma_{\mathfrak A}$ are equivalent (and therefore they are also equivalent to $\Gamma$.)

We note for future reference that if we instead consider the free countable \emph{unital} unnormed $\Q(i)$-$*$-algebra $\Au$ and let
$$
\hat \Gamma_{\Au}(H)=\{f:\Au\to\mathcal B(H):f\text{ is a \emph{unital} $*$-homomorphism}\},
$$
then this gives a parameterization of all unital separable \CSTAR subalgebras of $\mathcal B(H)$. Note that $\Au$ may be identified with the set of all formal $*$-polynomials in the variables $X_k$ with coefficients in $\Q(i)$ (allowing a constant term.)

\section{The Elliott invariant}\label{S.Elliott}

In this section we introduce a standard Borel space of 
Elliott invariants. We prove that the computation of the Elliott invariant of
$C^*(\gamma)$ is given by a Borel-measurable function (Theorem~\ref{T.Ell})).
The Elliott invariant of a unital \CSTAR algebra  $A$ is the sextuple
(see \cite{Ror:Classification}, \cite{RoLaLa:Introduction})
\[
(K_0(A), K_0(A)^+,[1_A]_0,K_1(A), T(A), r_A\colon T(A)\to S(K_0(A)).
\]
Here $(K_0(A),K_0(A)^+, [1_A]_0)$ is the ordered $K_0$-group with the canonical order unit, 
$K_1(A)$ is the $K_1$-group of $A$, and  $T(A)$ is the Choquet simplex of all tracial
states of $A$. Recall that a state $\phi$ on a unital \CSTAR algebra  $A$ is \emph{tracial} 
if $\phi(ab)=\phi(ba)$ for all $a$ and $b$ in $A$. Finally, 
 $r_A\colon T(A)\to S(K_0(A))$ is the coupling map that associates
a state on $K_0(A)$ to every trace on $A$. Recall that a \emph{state} on an ordered Abelian 
group is a positive homomorphism  $f\colon G\to (\bbR,+)$ and that the Murray-von Neumann
equivalence of projections $p$ and $q$ in $A$ implies $\phi(p)=\phi(q)$ for every 
trace $\phi$ on $A$ (see \S\ref{S.K} below).

\subsection{Spaces of  countable groups} 
A reader familiar with the logic actions may want to skip the following few paragraphs.   
As usual, identify $n\in\N$ with the set $\{0,1,\ldots, n-1\}$. For $n\in\N\cup\{\N\}$, let
$$
\bfS(n)=\{f:n^2\to n:(\forall i,j,k\in n) f(i,f(j,k))=f(f(i,j),k)\}.
$$
Note that $\bfS(\N)$ is closed when $\N^{\N^2}$ is given the product topology, and that if for $f\in\bfS(n)$  and $i,j\in n$ we define $i\cdot_f j$ as $f(i,j)$, then $\cdot_f$ gives $n$ a semigroup structure. The space $\bfS(n)$ may therefore be thought of as a Polish space parameterizing all countable semigroups with underlying set $n\in\N\cup\{\N\}$. We let $\bfS'(n)=\bfS(n)\times n$, and think of elements $(f,i)\in\bfS'(n)$ as the space of semigroups with a distinguished element $i$.

The subsets of $\bfS(n)$ (respectively $\bfS'(n)$) consisting of Abelian semigroups, groups and Abelian groups form closed subspaces that we denote by $\bfS_a(n)$, $\bfG(n)$ and $\bfG_a(n)$ (respectively $\bfS'_a(n)$, $\bfG'(n)$ and $\bfG_a'(n)$). The isomorphism relation in $\bfS(n)$, $\bfS_a(n)$, $\bfG(n)$ and $\bfG_a(n)$, as well as the corresponding ``primed'' classes, are induced by the natural action of the symmetric group $\Sym(n)$. These are very special cases of the logic actions, see \cite[2.5]{BeKe:Polish}.

We also define the spaces $\bfGord(n)$ and $\bfGordu(n)$ of \emph{ordered} Abelian groups and ordered Abelian groups with a distinguished \emph{order unit}, in the sense of \cite[Definition~1.1.8]{Ror:Classification}. The space $\bfGord(n)$ consists of pairs $(f,X)\in\bfG_a(n)\times \mathcal P(n)$ such that if we define for $x,y\in n$ the operation $x+_f y=f(x,y)$ and $x\leq_X y\iff y+_f(-x)\in X$, then we have $X+_fX\subseteq X$, $-X\cap X=\{0\}$ and $X-X=n$. The space $\bfGordu(n)$ consists of pairs $((f,X),u)\in \bfGord(n)\times n$ satisfying additionally the conditions
\begin{enumerate}
\item $u\in X$;
\item for all $x\in n$ there is $k\in\N$ such that $-ku\leq_X x\leq_X ku$.
\end{enumerate}
From their definition it is easy to verify that $\bfGord(\N)$ and $\bfGordu(\N)$ form $G_\delta$ subsets of $\bfG_a(\N)\times\mathcal P(\N)$ and $\bfGord(\N)\times\N$, and so are Polish spaces.

Define $\bfGord$ and $\bfGordu$ to be the disjoint unions
$$
\bfGord=\bigsqcup_{n\in\N\cup\{\N\}} \bfGord(n)\text{ and } \bfGordu=\bigsqcup_{n\in\N\cup\{\N\}}\bfGordu(n)
$$
and give these spaces the natural standard Borel structure. Similarly, define the standard Borel spaces $\bfS$, $\bfS_a$, $\bfG$ and $\bfG_a$ and their primed counterparts to be the disjoint union of their respective constituents.

\subsection{Traces and states} Recall that a compact convex set $K$ is a \emph{Choquet simplex} 
if for every point $x$ in $K$ there exists a unique probability measure $\mu$ supported by  the extreme boundary of $K$ that has $x$ as its barycentre   (see \cite{alfsen71}). Every metrizable Choquet simplex is affinely homeomorphic to a subset of $\Delta^{\bbN}$, 
with  $\Delta=[0,1]$. 

For every \CSTAR  algebra $A$ the space $T(A)$ of its traces 
is a Choquet simplex. In case when $A$ is separable 
  it can be identified with a compact convex subset of the Hilbert cube~$\Delta^{\bbN}$. 
In \cite[Lemma 4.7]{FaToTo2}
it was shown that  all Choquet simplexes form a Borel subset  of 
the $F(\Delta^{\bbN})$. 


A \emph{state} on an ordered Abelian group with unit $(G,G^+,1)$ is a homomorphism 
$\phi\colon G\to \bbR$ such that $\phi[G^+]\subseteq\bbR^+$ and $\phi(1)=1$. For every 
$n\in \bbN\cup \{\bbN\}$ the set $\cZ_0$ 
of all $((f,X,u),\phi) \in \bfGordu(n)\times \bbR^n$ such that 
$\phi[X]\subseteq \bbR^+$, $\phi(u)=1$ and $\phi(f(i,j))=\phi(i)+\phi(j)$ for all $i,j$ 
is clearly closed.   By \cite[Lemma 3.16]{FaToTo2}, 
  the map 
  \[
  \States\colon \bfGordu(n)\to \bbR^n
  \]
    such that $\States(f,X,u)$ is the 
   section of $\cZ_0$ at $(f,X,u)$ is Borel. 

Recall that $\Kconv$ denotes the compact metric space of compact convex subsets
of $\Delta^{\bbN}$. 

\begin{lemma} \label{L.dual.1.5} 
There is a continuous map  $\Psi\colon \Kconv\to C(\Delta^{\bbN}, \Delta^{\bbN})$ such that 
$\Psi(K)$ is a retraction from~$\Delta^{\bbN}$ onto $K$ for  
 all $K\in\Kconv$.
\end{lemma}

\begin{proof} Identify $\DeltaN$ with $\prod_n [-1/n,1/n]$  and consider the compatible 
$\ell_2$ metric $d_2$ on $\DeltaN$. 
Consider the set
\[
\cZ=\{(K,x,y): K\in \Kconv, x\in \DeltaN, y\in K, \text{ and  } d_2(x,y)=\inf_{z\in K} d_2(x,z)\}.
\]
Since the map $(K,x)\mapsto \inf_{z\in K} d_2(x,z)$ is continuous 
on $\{K\in \Kconv: K\neq\emptyset\}$, 
this set is closed. Also, for every $K,x$ there is the unique point $y$ such that 
$(K,x,y)\in \cZ$ (e.g., \cite[Lemma~3.1.6]{Pede:Analysis}). 
By compactness, the function $\chi$ that sends $(K,x)$ to the unique $y$ 
such that $(K,x,y)\in \cZ$ is 
continuous. Again by compactness, the map $\Psi(K)=\{(x,y): (K,x,y)\in \cZ\}$ is continuous. 
\end{proof}

For $K\in \Kconv$, $n\in \bbN\cup\{\N\}$,  and $(f,X,u)\in \bfGordu(n)$ let 
$\Pairing(K, (f,X,u))$ be the set of all $h\colon \Delta^{\bbN}\to \bbR^n$ such that 
there exists a continuous affine function $h'\colon K\to \States(X,f,u)$ such that with $\Psi$ as 
in Lemma~\ref{L.dual.1.5} the following diagram commutes
\[
\diagram 
K \drto_{h'}&  & \Delta^{\bbN} \llto_{\Psi(K)}\dlto^h \\
 & \bbR^n. 
\enddiagram
\]
Again the set of all $(K,(f,X,u),h)$ as above is closed and by \cite[Lemma 3.16]{FaToTo2} 
the map $\Pairing$ is Borel. 

   
\begin{definition}
The space $\bfEll$ of Elliott invariants is a subspace of
\[
\bfGordu\times \bfG_a\times \Kconv\times \bigsqcup_{n\in \bbN\cup \{\bbN\}} C(\Delta^{\bbN}, \Delta^n) 
\]
consisting of quadruples  $(G_0,G_1,T,r)$, where $G_0\in \bfGordu$, $G_1\in \bfG_a$, 
$T\in\Kconv$ is a Choquet simplex, and $r\in\Pairing(T,G_0)$. 
By the above and \cite[Lemma 4.7]{FaToTo2}, the set $\bfEll$ is  Borel and therefore it 
is a standard Borel space with the induced Borel structure. 

We say that two such quadruples $(G_0,G_1,T,r)$ and $(G_0',G_1',T',r')$ in $\bfEll$ are \emph{isomorphic} if $G_0\cong G_0'$, $G_1\cong G_1'$, and there is an affine isomorphism $\alpha: T\to T'$ such that we have $\hat \eta\rs T\circ r=r'\circ \alpha\rs T'$, where $\hat\eta:S(G_0)\to S(G_0')$ corresponds to some isomorphism $\eta: G_0\to G_0'$. This is clearly an analytic equivalence relation. 
\end{definition}

\subsection{Computing K-theory} 
\label{S.K}
The isomorphism  relation defined above is clearly analytic and it corresponds to the isomorphism of Elliott invariants. The rest of this section contains the proof of the following theorem.

\begin{theorem} \label{T.Ell}
There  is a Borel map  $\Ell\colon \Gammau\to \bfEll$ such that
$\Ell(\gamma)$ is the Elliott invariant of $C^*(\gamma)$, for all $\gamma\in\Gamma$.
\end{theorem}

We will start by showing:

\begin{prop} \label{P.K-0} There is a Borel map $\Kzerou\colon \Gammau\to \bfGordu$ such that
\[
\Kzerou(\gamma)\cong (K_0(C^*(\gamma)), K_0(C^*(\gamma))^+,[1_{C^*(\gamma)}]_0)
\]
for all $\gamma$.
\end{prop}

For a \CSTAR algebra $A$, let $\sim_A$ denote the \emph{Murray-von Neumann equivalence} of projections in $A$. Therefore $p\sim_A q$ if there is $v\in A$ such that $vv^*=p$ and $v^*v=q$. 
Note that $p\sim_A q$ implies $\phi(p)=\phi(q)$ for every trace $\phi$ of $A$. 
 If $A$ is clear from the context we will simply write $\sim$. Also, following the usual conventions, for $a,b\in\cB(H)$ we write $a\oplus b$ for the element
$$
\begin{pmatrix} a & 0 \\ 0 & b \end{pmatrix}\in M_2(\cB(H)).
$$

For the next Lemma, recall from \cite[Lemma 3.13]{FaToTo2} the Borel function $\Proj:\Gamma\to\Gamma$ which computes, for each $\gamma\in\Gamma$, a sequence of projections that are dense in the set of projections in $C^*(\gamma)$.

\begin{lemma}\label{L.K.1}
(1) The relation $\Upsilon_1\subseteq \Gamma\times\N\times\N$ defined by
$$
\Upsilon_1(\gamma,m,n)\iff \Proj(\gamma)(m) \sim_{C^*(\gamma)} \Proj(\gamma)(n)
$$
is Borel.

(2) The relation $\Upsilon_2\subseteq \Gamma\times\N\times\N\times\N$ defined by
$$
\Upsilon_2(\gamma,m,n,k)\iff \Proj(\gamma)(m)\oplus\Proj(\gamma)(n) \sim_{M_2(C^*(\gamma))} \Proj(\gamma)(k)\oplus 0
$$
is Borel.
\end{lemma}

\begin{proof}
To see (1), note that
$$
\Upsilon_1(\gamma,m,n)\iff (\exists k) \|\p_k(\gamma)\p_k(\gamma)^*-\Proj(\gamma)(m)\|<\frac 1 4\wedge \|\p_k(\gamma)^* \p_k(\gamma)-\Proj(\gamma)(n)\|<\frac 1 4.
$$
For (2), note that for $m,n,k\in\N$ the maps $\Gamma\to M_2(\cB(H))$:
$$
\gamma\mapsto \Proj(\gamma)(m)\oplus\Proj(\gamma)(n) \text{ and } \gamma\mapsto \Proj(\gamma)(k)\oplus 0
$$
are Borel by \cite[Lemma 3.7]{FaToTo2}. Thus
\begin{align*}
\Upsilon_2(m,n,k)\iff (\exists i)&\|\p_i(M_2(\gamma))\p_i(M_2(\gamma))^*-\Proj(\gamma)(m)\oplus \Proj(\gamma)(n)\|<\frac 1 4\\
\wedge&\|\p_i(M_2(\gamma)\p_i(M_2(\gamma))^*-\Proj(\gamma)(k)\oplus 0\|<\frac 1 4
\end{align*}
gives a Borel definition of $\Upsilon_2$.
\end{proof}

\begin{proof}[Proof of Proposition \ref{P.K-0}]
Note in Lemma~\ref{L.K.1} that for each $\gamma\in\Gammau$, the section $(\Upsilon_1)_\gamma=\{(m,n)\in\N:\Upsilon_1(\gamma,m,n)\}$ defines an equivalence relation denoted $\sim_\gamma$ on $\N$. Let $B_n\subseteq \Gammau$, $(n\in\N\cup\{\infty\}$), be the set of $\gamma\in\Gamma$ such that $(\Upsilon_1)_\gamma$ has exactly $n$ classes. Then $(B_n)$ is a Borel partition of $\Gammau$. On each $B_n$ we can find Borel functions $\sigma_{n,i}:B_n\to\N$, ($0\leq i< n$), selecting exactly one point in each $(\Upsilon_1)_\gamma$-class. Identifying $n\in\N$ with the set $\{0,\ldots, n-1\}$, let $V_0(\gamma)$ (where $\gamma\in B_n$) be the semigroup on $n$ defined by
$$
i+j=k\iff \Upsilon_2(\gamma,\sigma_{n,i}(\gamma),\sigma_{n,j}(\gamma),\sigma_{n,k}(\gamma)).
$$
By \cite[Lemma 3.10]{FaToTo2} there is a Borel map $\Psi:\Gamma\to\Gamma$ such that $C^*(\Psi(\gamma))\simeq C^*(\gamma)\otimes\cK$. We define $V(\gamma)=V_0(\Psi(\gamma))$ and note that this gives us a Borel assignment $B_n\to\bfS(n)$ of semigroup structures on $n$. The $\Kzero$ group of $C^*(\gamma)$ is then the Grothendieck group constructed from $V(\gamma)$ with the order unit being the unique $i\in n$ such that $\sigma_{n,i}(\gamma)\sim_\gamma u(\gamma)$, and so the proof is complete once we prove the next Lemma.
\end{proof}

\begin{lemma}
There is a Borel map $\bfS_a\to\bfGord$ associating to each Abelian semigroup (defined by) $f\in \bfS_a$ the Grothendieck group constructed from $f$.
\end{lemma}
\begin{proof}
It is enough to construct a Borel $\bfS_a(n)\to\bfGord$ as required for each $n\in\N\cup\{\N\}$. We follow the description of the Grothendieck group given in \cite[V.1.1.13]{blackadar}. Defining
$$
P=\{(f,(i,j),(k,l))\in\bfS_a\times n^2\times n^2: (\exists m) i+_f l+_f+m=k+_f j+_f m\},
$$
we have that that the section $P_f$ is an equivalence relation on $n^2$ for all $f\in \bfS_a$. Write $\bfS_a(n)$ as a disjoint union of Borel pieces $B_k$ ($k\in\N\cup\{\N\}$)such that $f\in B_k$  if and only if $P_f$ has exactly $k$ classes. We can then find on each piece $B_k$ Borel functions selecting an element in each $P_f$ class, and from this selection the Grothendieck group of $f$ can be defined on $k$ in a Borel way.
\end{proof}

\begin{corollary}
There is a Borel map $\Kzero:\Gamma\to\bfGord$ such that
$$
\Kzero(\gamma)\simeq (K_0(C^*(\gamma)),K_0^+(C^*(\gamma)))
$$
\end{corollary}

\begin{proof}
By \cite[Lemma 3.12]{FaToTo2} the unitization $\tilde C^*(\gamma)$ of $C^*(\gamma)$ is obtained via a Borel function, and by the above proof so is $K_0(\tilde
C^*(\gamma))$. Then $K_0(C^*(\gamma))$ is isomorphic to the quotient
of $K_0(\tilde C^*(\gamma))$ by its subgroup generated by the image
of the identity in $\tilde C^*(\gamma)$. 
\end{proof}

\begin{prop}
\label{P.K-1} There is a Borel map $\Kone\colon \Gamma\to \bfG$ such
that
\[
\Kone(\gamma)\cong K_1(C^*(\gamma))
\]
for all $\gamma$.
\end{prop}

\begin{proof} By  Bott periodicity, $K_1(C^*(\gamma))\cong
K_0(C((0,1),  A))$ and by \cite[Lemma 3.11]{FaToTo2} and
Proposition~\ref{P.K-0} the right-hand side can be computed by a
Borel function.
\end{proof}

\begin{proof}[Proof of Theorem~\ref{T.Ell}]
The computation of $K$-theory  is  Borel by Proposition~\ref{P.K-0} and
Proposition~\ref{P.K-1}. By  \cite[Lemma 3.17]{FaToTo2} the computation of the tracial
simplex $\bbT(\gamma)\cong T(C^*(\gamma))$ is Borel as well. Since $\phi\in \bbT(\gamma)$
is identified with a continuous map on a dense subset of $C^*(\gamma)$, by restricting this map
to $\Proj(\gamma)$ and then composing with the embedding of $\Proj(\gamma)$
into $K_0(\gamma)$ we obtain the restriction of the coupling map $r_{C^*(\gamma)}$ to the
positive part of $K_0(\gamma)$. The coupling map is now canonically extended to $K_0(\gamma)$. 
\end{proof} 

In Section 4 of \cite{FaToTo2} we defined an alternative space of Choquet simplexes and 
showed that it is weakly equivalent to the more straightforward one used 
above.


\section{The Cuntz semigroup}
\label{S.Cuntz}

In this section we prove that the Cuntz semigroup of a separable \CSTAR algebra is Borel computable, as is a related invariant, the radius of comparison.  The relevance of the Cuntz semigroup to \CSTAR algebra classification was demonstrated in \cite{Toms}, where it was used to distinguish simple unital separable nuclear \CSTAR algebras with identical Elliott invariants;  see also \cite{tcmp}.  We review the basic properties of Cuntz semigroups below.  A good general reference is \cite{apt}.

\subsection{Basics and sup-dense sets}\label{cuntzbasics}
Let $A$ be a \CSTAR algebra. Let us consider on
$(A\otimes\mathcal K)_+$ the relation $a\precsim b$ if $v_nbv_n^*\to a$ for some sequence $(v_n)$ in $A\otimes\mathcal K$.
Let us write $a\sim b$ if $a\precsim b$ and $b\precsim a$. In this case we say that $a$ is Cuntz equivalent to $b$.
Let $Cu(A)$ denote the set $(A\otimes \mathcal K)_+/\sim$ of Cuntz equivalence classes. We use $[a]$ to denote the class of $a$ in $Cu(A)$.  It is clear that
$[a] \leq [b] \Leftrightarrow a\precsim b$ defines an order on $Cu(A)$. We also endow $Cu(A)$
with an addition operation by setting $[a]+[b]:=[a'+b']$, where
$a'$ and $b'$ are orthogonal and Cuntz equivalent to $a$ and $b$ respectively (the choice of $a'$ and $b'$
does not affect the Cuntz class of their sum). 

The semigroup $Cu(A)$ is an object in a category of ordered Abelian monoids denoted by ${\mathbf{Cu}}$, a category in which the relation of order-theoretic compact containment plays a significant role, see \cite{cei}.  Let $T$ be a preordered set with $x,y \in T$.  We say that $x$ is compactly contained in $y$---denoted by $x \ll y$---if for any increasing sequence $(y_n)$ in $T$ with supremum $y$, we have $x \leq y_{n_0}$ for some $n_0 \in \mathbb{N}$.   
An object $S$ of $\mathbf{Cu}$ enjoys the following properties:
\begin{enumerate}[\indent $\bf P 1$]
\item $S$ contains a zero element; \vspace{1mm}
\item the order on $S$ is compatible with addition:  $x_1 + x_2
\leq y_1 + y_2$ whenever $x_i \leq y_i , i \in \{1, 2\}$; \vspace{1mm}
\item every countable upward directed set in $S$ has a supremum; \vspace{1mm}
\item the set $x_\ll = \{y \in S \ |  \ y \ll x\}$ is nonempty and upward directed with respect to both
$\leq$ and $\ll$, and contains a sequence $(x_n)$ such that
$x_n \ll x_{n+1}$ for every $n \in \mathbb{N}$ and $\sup_n x_n = x$; \vspace{1mm}
\item the operation of passing to the supremum of a countable upward directed set and the relation
$\ll$ are compatible with addition:  if $S_1$ and $S_2$ are countable upward directed sets in $S$, then
$S_1 + S_2$ is upward directed and $\sup (S_1 + S_2) = \sup S_1 + \sup S_2$, and if $x_i \ll y_i$ for $i \in \{1, 2\}$, then
$x_1 + x_2 \ll y_1 + y_2$ .
\end{enumerate}

\noindent
Here we assume further that $0 \leq x$ for any $x \in S$.  This is always the case for $Cu(A)$.
For $S$ and $T$ objects of $\textbf{Cu}$, the map $\phi\colon S\to T$ is a morphism in the category $\mathbf{Cu}$ if
\begin{enumerate}[\indent $\bf M 1$]
\item $\phi$ preserves the relation $\leq$; \vspace{1mm}
\item $\phi$ is additive and maps $0$ to $0$; \vspace{1mm}
\item $\phi$ preserves the suprema of increasing sequences; \vspace{1mm}
\item $\phi$ preserves the relation $\ll$. \vspace{1mm}
\end{enumerate}

\begin{definition} Let $S \in \mathbf{Cu}$.  A countable subset $D$ of $S$ is said to be \emph{sup-dense} if each $s \in S$ is the supremum of a $\ll$-increasing sequence in $D$.  We then say that $S$ is countably determined.  (Here by $\ll$ we mean the relation $\ll$ on $D$ inherited from $S$, i.e., $d_1\ll d_2$ in $D$ iff $d_1 \ll d_2$ in $S$.)
\end{definition}

\begin{definition}\label{Cu0}
Let $\Cu_0$ denote the category of pairs $(S,D)$, where $S$ is a countably determined 
element of $\mathbf{Cu}$, and $D$ is a distinguished sup-dense subset of $S$ which is moreover a semigroup with the binary operation inherited from $S$.  We further assume $D$ to be equipped with the relations $\leq$ and $\ll$ inherited from $S$.
\end{definition}

An element $x$ of $S \in \mathbf{Cu}$ such that $x \ll x$ is \emph{compactly contained in itself}, 
or briefly \emph{compact}.  
If $(S,D) \in \Cu_0$ then $D$ automatically contains all compact elements.



\subsection{A standard Borel space for $\mathbf{Cu}$}
Let $\BCu$ be the space of triples $(\oplus,\lCu,\ll)$ in $\N^{\N\times\N}\times\mathcal P(\N\times\N)\times\mathcal P(\N\times\N)$ with the following properties:
\begin{enumerate}
\item $(\bbN,\oplus, \lCu,\ll)$ is an 
ordered semigroup under the order $\lCu$ (we will use $a,b,c\ldots$ to represent elements of the semigroup);
\item  $\ll$ is a transitive antisymmetric relation with the property that $a \ll b$ and $c \ll d$ implies $a \oplus c \ll b \oplus d$;
\item $a \ll b$ implies $a \lCu b$;
\item for each $a$ in the semigroup, there is some $b \ll a$, and if $a$ does not satisfy $a \ll a$, then the set of all such $b$ is upward directed and has no maximal element.
\end{enumerate}
({\bf Warning:}  $\ll$ here is {\bf not} defined in terms of $\lCu$ as in our discussion of the 
 Cuntz semigroup, but rather is just some other relation finer than $\leq$.  It will coincide with the Cuntz semigroup definition in the case that an element of $\BCu$ really is a sup-dense subsemigroup of an element of the category $\mathbf{Cu}$.)
We can define a map $\Phi:\Cu_0\to \BCu$ in an obvious way:  send $(S,D)$ to the triple $(\oplus, \lCu, \ll)$ corresponding to $D$ ($D=\{d_n: n\in \bbN\}$ is  the ordered semigroup on $\bbN$
defined by $m\oplus n=k$ if and only if $d_m + d_n = d_k$,  $m\lCu n$ if and only if 
$d_m \lCu d_n$ and $m\ll n$ if and only if $d_m\ll d_n$). 

If $D \in \BCu$, we let $D^\nearrow$ denote the set of $\ll$-increasing sequences in $D$.  Define an equivalence relation on $\approx$ on $D^\nearrow$ by
\[
(x_n) \approx (y_n) \Leftrightarrow (\forall m)(\exists n) \ x_m \ll y_n \ \mathrm{and} \ y_m \ll x_n.
\]
Equip $D^\nearrow$ with the relations 
\[
(x_n) \leq^\nearrow (y_n) \Leftrightarrow (\forall n)(\exists m) x_n \leq y_m
\]
and
\[
(x_n) \ll^\nearrow (y_m) \Leftrightarrow (\exists m_0)(\forall n) x_n \lCu y_{m_0}.
\]
Notice that 
\[
(x_n) \leq^\nearrow (y_n) \wedge (y_n) \leq^\nearrow (x_n) \Leftrightarrow (x_n) \approx (y_n).
\]   
Define $(x_n) \oplus^\nearrow (y_n) = (x_n \oplus y_n)$, and set 
\[
W(D) = D^\nearrow/\approx\qquad\text{and}\qquad W(S)= S^\nearrow/\approx.
\]
  Note that the operation $\oplus$ and the relations $\leq^\nearrow$ and $\ll^\nearrow$ drop to an operation $+$ and relations $\precsim$ and $\ll$ on $W(D)$, respectively. 

If $(S,D)$ in $\Cu_0$, then the semigroup $D$ is an element of the category {\bf PreCu} introduced in \cite{abp:comp}, and $S$ is a completion of $D$ in {\bf Cu} in the sense of \cite[Definition 3.1]{abp:comp}.  An appeal to \cite[Lemma 4.2]{abp:comp} and \cite[Theorem 4.8]{abp:comp} shows that $S$, too, is a completion of $D$ in {\bf Cu}, whence
\[
W(D) \cong W(S) \cong S
\]
in {\bf Cu}.

\begin{remark}
The space $\BCu$ certainly does not consist only of sup-dense subsemigroups of elements of $\mathbf{Cu}$, but it does contain all such sets.  It is not clear how one might define a standard Borel space of all sup-dense subsemigroups of elements of $\mathbf{Cu}$, but $\BCu$ will suffice for our purposes.
\end{remark}

\subsection{The structure and morphisms of $W(D)$}
  Let $\cY$ be the Borel space of all functions from $\bbN$ to the Baire space $\bbN^\bbN$. 
Since $\alpha\in \cY$ is a map from $\bbN\to \bbN^{\bbN}$ and 
the elements of $\BCu$ have $\bbN$ as the underlying set, 
if $D_1$ and $D_2$ in $\BCu$ are fixed then $\alpha\in \cY$ codes a map from 
$D_1$ to $D_2^{\bbN}$.  We shall identify $\alpha$ with this map whenever $D_1$ and $D_2$
are clear from the context.  
 The set of all triples $(D_1,D_2,\alpha)$ such that the range of $\alpha$ is included in $D_2^\nearrow$  is a closed subset of $\BCu^2\times \cY$.  
To each $D \in \BCu$ we associate a map $\eta_D:D \to D^\nearrow$ (or simply $\eta$ if $D$ is clear from the context) as follows: Select, in a Borel manner, a sequence $\eta_D(a) = (a_n)$ which is cofinal in $\{b \in D \ | \ b \ll a \}$.  The association $D \mapsto \eta_D$ is then Borel.
If $F\colon W(D_1)\to W(D_2)$ is a semigroup homomorphism preserving $\leq$ and $\ll$ and $\alpha \in \cY$ then 
 we say that $\alpha$ \emph{codes  $F$}
 if $[\alpha(a)]=F[\eta(a)]$ for all $a\in D_1$. Note that $\alpha$ really codes the restriction of 
 $F$ to $\eta(D_1)$, but, as we shall see, $F$ is completely determined by this restriction if $W(D_1)$ and $W(D_2)$ are in the category $\mathbf{Cu}$.

\begin{lemma}\label{wayless}
Let $(S,D) \in \Cu_0$.  Then $(a_n) \ll^\nearrow (b_n)$ in $D^\nearrow$ if and only if $[(a_n)] \ll [(b_n)]$ in $W(D) \cong S$, where $\ll$ is the relation of order-theoretic compact containment inherited from the relation $\leq$ on $S$ (see Subsection \ref{cuntzbasics}).
\end{lemma}

\begin{proof}
Suppose first that $(a_n) \ll^\nearrow (b_m)$, and fix $m_0$ such that $a_n \lCu b_{m_0}$ for all $n$.  We must prove that if, for fixed $j$, $(c_n^j) \in D^\nearrow$, and if moreover $[(c_n^j)]$ is a $\leq$-increasing sequence in $j$ with supremum $[(b_n)]$, then $[(a_n)] \leq [(c_n^{j_0})]$ for some $j_0 \in \mathbb{N}$.  First we recall (see the proof of the existence of suprema in inductive limits of Cuntz semigroups in \cite{cei}) that for such $(c_n^j)$, there is a sequence of natural numbers $(n_j)$ with the property that $(c_{n_j}^j) \approx (b_m)$.  In particular, there is $j_0$ such that $b \ll c_{n_{j_0}}^{j_0}$.  Since $(c_k^{j_0}) \in D^\nearrow$, we have 
\[
a_n \ll b_{m_0} \ll c_{n_{j_0}}^{j_0}
\]
for all $n\in\N$, and so $(a_n) \ll^\nearrow (c_k^{j_0})$.  This implies $[(a_n)] \leq [(c_k^{j_0})]$, as required.

For the converse, assume that $[(a_n)] \ll [(b_m)]$.  Since $b_m \ll b_{m+1}$, we know that for any element of the sequence $\eta(b_m)$, there is an element of the sequence $\eta(b_{m+1})$ that $\ll$-dominates it, so that $[\eta(b_m)] \leq [\eta(b_{m+1})]$.  There is also, for given $m$, and element of the sequence $\eta(b_{m+1})$ that $\ll$-dominates $b_m$.  (These two assertions follow from property (4) in the definition of $\BCu$.)  It follows that for some sequence $m_j$, we have $[(\eta(b_j)_{m_j}] \geq [(b_m)]$.  Identifying $\eta(b_j)$ with $(c_n^j)$ from the first part of the proof and observing (see again the proof of existence of suprema in inductive limits of Cuntz semigroups in \cite{cei}) that the $n_j$ chosen above can be increased without disturbing the fact that $(c_{n_j}^j) \approx (b_n)$, we see that by increasing the $m_j$ if necessary, we also have that $\sup_j [\eta(b_j)] = [\eta(b_j)_{m_j}] \geq [(b_m)]$.  It follows that $[\eta(b_{j_0})] \geq [(a_n)]$ for some $j_0$, whence $a_n \ll b_{j_0}$ for all $n$, as required.
\end{proof}
 
 \begin{lemma}\label{waylessembed}
Let $(S,D) \in \Cu_0$. Then $a \ll b$ in $D$ if and only if $[\eta(a)] \ll [\eta(b)] \in W(D) \cong S$.
\end{lemma}

\begin{proof}
By Lemma \ref{wayless}, it will suffice to prove that $a \ll b$ iff $\eta(a) \ll^\nearrow \eta(b)$ in $D^\nearrow$.  Suppose first that $a \ll b$, so that $(\eta(a)_i)$ and $(\eta(b)_m)$ are $\ll$-increasing sequences in $D$ with suprema $a$ and $b$, respectively (this uses several facts: that $D$ is embedded in $S \in \mathbf{Cu}$; that objects in $\mathbf{Cu}$ admit suprema for increasing sequences; and that $S$ may be identified with $W(S)$ as in Subsection \ref{cuntzbasics}).  Since $a \ll b$, there is $m_0$ such that $\eta(b)_m \geq a \geq \eta(a)_i$ for all $i$ and for all $m \geq m_0$, as required.

Conversely, suppose that $\eta(a) \ll^\nearrow \eta(b)$, so that there is $m_0$ such that $\eta(a)_i \leq \eta(b)_m$ for all $m \geq m_0$.  Now
\[
\sup_i(\eta(a)_i) = a \leq \eta(b)_m \ll \eta(b)_{m+1} \leq b,
\]
so that $a \ll b$.
\end{proof}

\begin{lemma}\label{less}
Let $(S,D) \in \Cu_0$. Then $(a_n) \leq^\nearrow (b_n)$ in $D^\nearrow$ if and only if $[(a_n)] \leq [(b_n)]$ in $W(D) \cong S$.
\end{lemma}

\begin{proof}
Suppose that $(a_n) \leq^\nearrow (b_n)$.  It follows that for each $n \in \mathbb{N}$ there is $m(n)$ such that 
\[
a_n \leq b_{m(n)} \ll b_{m(n)+1}.
\]  
The statement $[(a_n)] \leq [(b_n)]$ amounts to the existence of $(c_n) \in D^\nearrow$ such that $(a_n) \approx (c_n)$ and $(c_n) \leq^\nearrow (b_n)$.  Here we can take $(c_n) = (a_n)$, completing the forward implication.

Suppose, conversely, that $[(a_n)] \leq [(b_n)]$, so that there is some $(c_n) \in D^\nearrow$ such that $(a_n) \cong (c_n)$ and $(c_n) \leq^\nearrow (b_n)$.  Since $(a_n)$ and $(c_n)$ are cofinal in each other with respect to $\ll$, it is immediate that $(a_n) \leq^\nearrow (b_n)$.
\end{proof}

\begin{lemma}\label{lessembed}
Let $(S,D) \in \Cu_0$. Then $a \leq b$ in $D$ if and only if $[\eta(a)] \leq [\eta(b)]$ in $W(D) \cong S$.
\end{lemma}

\begin{proof}
By Lemma \ref{less}, it is enough to prove that $a \leq b$ iff $\eta(a) \leq^\nearrow \eta(b)$ in $D^\nearrow$.  Suppose first that $a \leq b$.  The sequence $(\eta(a)_n)$, being cofinal with respect to $\ll$ in $\{c \in D \ | \ c \ll a \}$, has a supremum in $S$, namely, $a$ itself.  A similar statement holds for $b$.  For any $n \in \mathbb{N}$ we have $\eta(a)_n \ll a$, and $\sup \eta(b)_m = b \geq a$.  It follows that $\eta(b)_m \gg \eta(a)_n$ for all $m$ sufficiently large, whence $[\eta(a)] \leq [\eta(b)]$, as desired.

Suppose, conversely, that $\eta(a) \leq^\nearrow \eta(b)$ in $D^\nearrow$.  Since $\sup \eta(a)_n = a$, $\sup \eta(b)_m = b$, and for each $n$ there is $m$ such that $\eta(a)_n \ll \eta(b)_m$, it is immediate that $a \leq b$ in $S$.
\end{proof}

Using methods similar to those of Lemmas \ref{wayless}--\ref{lessembed} one can also prove the following result.

\begin{lemma}\label{embedD}
Let $(S,D) \in \Cu_0$.  Then the following are equivalent:
\begin{itemize}
\item[(i)] $a \oplus b = c$ in $D$;
\item[(ii)] $\eta(a) \oplus^\nearrow \eta(b) = \eta(c)$ in $D^\nearrow$;
\item[(iii)] $[\eta(a)] + [\eta(b)] = [\eta(c)]$ in $W(D)$.
\end{itemize}
\end{lemma}



\begin{lemma}\label{cuntzcode}
Let $D_1,D_2 \in \BCu$ be sup-dense subsemigroups of elements of $\mathbf{Cu}$. If $\alpha$ codes a homomorphism $\Phi:W(D_1) \to W(D_2)$ then for $a,b \in D_1$ we have:
\begin{enumerate}
\item $a\lCu b$ implies $(\forall m)( \exists n) \alpha(a)_m\lCu \alpha(b)_n$;
\item \label{C.Code.2} 
 $a\ll b$ implies $(\exists n)(\forall m) \alpha(a)_m\lCu \alpha(b)_n$;
\item $\alpha(a) \oplus \alpha(b)$ (defined pointwise) satisfies $\alpha(a) \oplus \alpha(b) \approx \alpha(a \oplus b)$.
\end{enumerate}
Conversely, if $\alpha$ has properties (1)--(3), then 
\[
\Psi: \left( \eta(D_1)/\approx \right) \cong D_1 \to W(D_2)
\]
given by $\Psi[(\eta(a)] = [\alpha(a)]$ is a homomorphism preserving $\leq$ and $\ll$, and if we let $\Psi[(a_n)] = \sup [\alpha(a_n)]$ for each $(a_n) \in D_1^\nearrow$ then this extends $\Psi$ to a homomorphism in $\mathbf{Cu}$ from $W(D_1) \to W(D_2)$.  

\end{lemma}

\begin{proof}
Let $\Phi:W(D_1) \to W(D_2)$ be a homomorphism, and assume that $\alpha$ codes $\Phi$.  We will prove (1);  the proofs of (2) and (3) are similar.  Let $a,b \in D_1$ satisfy $a \precsim b$.  By Lemma \ref{lessembed}, this is equivalent to $[\eta(a)] \leq [\eta(b)]$ in $W(D_1)$.  Since $\Phi$ is a homomorphism in $\mathbf{Cu}$, we have $\Phi[\eta(a)] \leq \Phi[\eta(b)]$, and since $\alpha$ codes $\Phi$, we conclude that $[\alpha(a)] \leq [\alpha(b)]$.   This last inequality is equivalent to the conclusion of (1) by Lemma \ref{less}.

Suppose now that $\alpha$ has properties (1)--(3).  The isomorphism  $ \left( \eta(D_1)/\approx \right) \cong D_1$ follows from Lemmas \ref{wayless}--\ref{embedD}.  Let us check that $\Psi$ preserves $\leq$ on $\eta(D_1)/\approx$ (verifying the other properties required of $\Psi$ on this domain is similar.)  First, we check that $\Psi$ is well-defined on $W(D_1)$.  Suppose $[(c_m)] = [(b_n)]$ in $W(D_1)$.  Passing to subsequences we have
\[
b_{n_1} \ll c_{m_1} \ll b_{n_2} \ll c_{m_2} \ll \cdots
\]
so that by property (2) and Lemmas \ref{wayless} and \ref{waylessembed} we have
\[
[\alpha(b_{n_1})] \ll [\alpha(c_{m_1})] \ll [\alpha(b_{n_2})] \ll [\alpha(c_{m_2})] \ll \cdots.
\]
It follows that 
\[
\sup_i [\alpha(b_{n_i})] = \sup_j [\alpha(c_{m_j})],
\]
Since 
\[
\sup_i [\alpha(b_{n_i})] = \sup_n [\alpha(b_n)],
\]
we have 
\[
\sup_n [\alpha(b_n)] = \sup_m [\alpha(c_m)]
\]
as required.  

To see that $\leq$ is preserved by $\Phi$ on $W(D_1)$, consider $[(b_n)] \leq [(c_m)]$.  Passing to subsequences we can assume that $b_k \ll c_k$ for every $k$.  Then by property (2) and Lemmas \ref{wayless} and \ref{waylessembed} we have
\begin{eqnarray*}
&& [\alpha(b_k)] \ll [\alpha(c_k)] \\
& \Rightarrow &\sup_k [\alpha(b_k)] \leq \sup_k [\alpha(c_k)] \\
& \Rightarrow & \Psi[(b_n)] \leq \Psi[(c_m)].
\end{eqnarray*}
 
\end{proof} 
 
\subsection{An analytic relation and isomorphism in $\Cu$} We shall now define an analytic equivalence relation on $\BCu$ which, for sup-dense subsemigroups of elements of $\mathbf{Cu}$, amounts to isomorphism.  Consider the standard Borel space $\BCu^2\times \cY^2$. 
In this space  consider the  Borel set $\cX$ consisting of all 
quadruples $(D_1,D_2, \alpha_1, \alpha_2)$
such that
\begin{enumerate}
\item  $\alpha_1$ and $\alpha_2$ satisfy the (Borel) conditions (1)--(3) of Lemma \ref{cuntzcode}
\item  $(\forall a\in D_1)(\forall b\in D_2)$ we have
\[
\alpha_1(a)\ll^\nearrow \eta(b) \ \Leftrightarrow \ \eta(a) \ll^\nearrow \alpha_2(b)
\]
and
\[
\eta(b) \ll^\nearrow \alpha_1(a) \ \Leftrightarrow \ \alpha_2(b) \ll^\nearrow \eta(a).
\]
\end{enumerate}
It is straightforward to verify that the conditions above define a Borel subset of $\BCu^2 \times \cY^2$, whence $\cX$ is a standard Borel space.
Now define a relation $E$ on $\BCu$ by
 \[
 D_1ED_2 \Leftrightarrow (\exists\alpha_1,\alpha_2) (D_1,D_2,\alpha_1,\alpha_2) \in \cX,
 \]
whence $E$, as the co-ordinate projection of $\cX$ onto $\BCu^2$, is analytic.

\begin{prop} \label{P.Cu.1} 
Let $D_1, D_2 \in \BCu$ be sup-dense subsemigroups of elements of $\mathbf{Cu}$.  It follows that $D_1ED_2$ iff $W(D_1) \cong W(D_2)$ in the category $\mathbf{Cu}$.
\end{prop}
\noindent

\begin{proof}
Assume $W(D_1)\cong W(D_2)$ and let $\Phi\colon W(D_1)\to W(D_2)$ be an  isomorphism. 
Pick $\alpha_1$ that codes $\Phi$ and $\alpha_2$ that codes $\Phi^{-1}$, so that $\alpha_1$ and $\alpha_2$ have the properties (1)--(3) of Lemma \ref{cuntzcode}. 
For (2) in the definition of $\cX$ we will only prove the first equivalence, as the second one is similar.  By Lemma \ref{wayless},  the first equivalence in (2) is equivalent to 
\[ 
(\forall a\in D_1)(\forall b\in D_2) \ \ [\alpha_1(a)] \ll [\eta(b)] \Leftrightarrow [\eta(a)] \ll [\alpha_2(b)].
\]
Suppose $[\alpha_1(a)] \ll [\eta(b)]$, so that 
\[
\Phi^{-1}[\alpha_1(a)] \ll \Phi^{-1}[\eta(b)]
\]
(morphisms in $\mathbf{Cu}$ preserve $\ll$).  Since $\alpha_2$ codes $\Phi^{-1}$, the right hand side above can be identified with $[\alpha_2(b)]$.  Similarly,
$\Phi^{-1}[\alpha_1(a)] = \Phi^{-1}\Phi[\eta(a)] = [\eta(a)]$,  so that $[\eta(a)] \ll [\alpha_2(b)]$.  The other direction is similar, establishing (2) from the definition of $\cX$, whence $D_1ED_2$.

Now assume $(D_1,D_2,\alpha_1,\alpha_2)\in \cX$ for some $\alpha_1$ and $\alpha_2$. 
Using Lemma \ref{cuntzcode} we obtain homomorphisms $\Phi_1:W(D_1)\to W(D_2)$ and $\Phi_2:W(D_2) \to W(D_1)$.   
Let us verify that $\Phi_2 \circ \Phi_1 = \mathrm{id}_{W(D_1)}$ (the proof for that $\Phi_1 \circ \Phi_2 = \mathrm{id}_{W(D_2)}$ is similar).  Fix $[(f_n)] \in W(D_1)$.  Since $a \mapsto [\eta(a)]$ is a complete order embedding of $D_1$ into $W(D_1)$ relative to $\precsim$ and $\ll$ by Lemmas \ref{wayless}--\ref{embedD}, we have $[(f_n)] = \sup_n [\eta(f_n)]$.  Since $\Phi_1$ preserves $\ll$, we have a corresponding $\ll$-increasing sequence $\Phi_1[\eta(f_n)] = [\alpha(f_n)]$, $i \in \mathbb{N}$ (see Lemma \ref{cuntzcode}).  Choose a $\ll$-increasing sequence $[\eta(b_i)]$ in $W(D_2)$ with supremum $\Phi_1[(f_n)]$, and note that this is also the supremum of the sequence $[\alpha_1(f_n)]$.  Since $W(D_2) \in \mathbf{Cu}$, we may, passing to a subsequence if necessary, assume that
\[
[\eta(b_i)] \ll [\alpha_1(f_i)] \ \ \mathrm{and} \ \  [\alpha_1(f_i)] \ll [\eta(b_{i+1})].
\]
Using (2) in the definition of $\cX$ and the relations above we obtain
\[
[\alpha_2(b_i)] \ll [\eta(f_i)] \ \ \mathrm{and} \ \ [\eta(f_i)] \ll [\alpha_2(b_{i+1})],
\]
so that the sequences $[\alpha_2(b_i)]$ and $[\eta(f_i)]$ have the same supremum, namely, $[(f_n)]$.  Now we compute:
\begin{eqnarray*}
(\Phi_2 \circ \Phi_1) [(f_n)] & = & (\Phi_2 \circ \Phi_1) \sup_i[\eta(f_i)] \\
& = & \Phi_2  \left( \sup_i \Phi_1[\eta(f_i)] \right) \\
& = & \Phi_2 \left( \sup_i [\alpha_1(f_i)] \right) \\
& = & \Phi_2 \left( \sup_i [\eta(b_i)] \right) \\
& = & \sup_i \Phi_2[\eta(b_i)] \\
& = & \sup_i [\alpha_2(b_i)] \\
& = & \sup_i [\eta(f_i)] \\
& = & [(f_n)].
\end{eqnarray*}

\end{proof}

Recall the following well-known lemma. 

\begin{lemma} For any strictly decreasing sequence $(\epsilon_n)$ of positive tolerances converging to zero, the sequence
$\langle (a-\epsilon_n)_+ \rangle$ is $\ll$-increasing in $Cu(A)$. \qed
\end{lemma}

In some cases, for example when $a$ is a projection, the sequence in the Lemma is eventually constant, i.e., $\langle a \rangle$ is compact.  
 This occurs, for instance, when $a \precsim (a-\epsilon)_+$ for some $\epsilon>0$.

\begin{prop} \label{P.Cu.2} There is a Borel map $\Psi\colon \Gamma\to \BCu$ such that $W(\Psi(\gamma)) \cong \Cu(C^*(\gamma))$.

\end{prop}

\begin{proof}
Fix $\gamma_0\in \Gamma$ such that $C^*(\gamma_0)$ is the algebra of compact operators 
and a bijection $\pi$ between $\bbN^2$ and $\bbN$. 
Moreover, choose $\gamma_0$ so that all operators in $\gamma_0$ 
have finite rank and $\gamma_0$ is closed under finite permutations of a fixed basis $(e_n)$ 
of $H$. 
We also fix a sequence of compact partial isometries $v_m$, such that $v_m$ swaps 
the first $m$ vectors of $(e_n)$ with the next $m$ vectors of this basis.  This sequence will 
be used in the proof of Claim~\ref{Cl.Cu.w} below.

Let $\Psi$ denote the Borel map from $\Gamma$ to $\Gamma$ obtained as the composition of
three Borel maps:  $\Tensor(\cdot,\gamma_0)$, where $\Tensor$ is the Borel map 
from \cite[Lemma 3.10]{FaToTo2};  the map $\gamma \mapsto  (\p_n(\gamma))$ (see \cite[Proposition 2.7]{FaToTo2});  and finally the map that sends $(a_n)$ to $(b_n)$ where 
\begin{equation}
b_n=((a_{\pi_0(n)}a_{\pi_0(n)}^*)-1/\pi_1(n))_+
\label{P.Cu.1.Eq.1}
\end{equation}
(here $n\mapsto (\pi_0(n),\pi_1(n))$ is the fixed bijection between $\bbN$ and $\bbN^2$). 

Fix $\gamma\in \Gamma$. 
Then $\gamma_1=\Tensor(\gamma,\gamma_0)$ satisfies $C^*(\gamma)\otimes \cK\cong C^*(\Tensor(\gamma,\gamma_0))$. 
Moreover, for  any two positive entries $a$ and $b$ of 
 $\gamma_1$ there are orthogonal positive $a'$ and $b'$ in $\gamma_1$ such that $a \sim a'$ and $b \sim b'$.  (Here $\sim$ denotes Cuntz equivalence.)  If $\gamma_2 = (\p_n(\gamma_1))$, then the elements of $\gamma_2$ are norm-dense in $C^*(\gamma_2) \cong C^*(\gamma) \otimes \cK$, and $\gamma_2$ contains $\gamma_1$ as a subsequence.  Finally, if  $\gamma_3$ is the sequence as in \eqref{P.Cu.1.Eq.1}
then $\gamma_3$ is a norm-dense set subset of the positive elements of $C^*(\gamma)\otimes \cK$. 
Let us write $d_m(\gamma):=(\gamma_3)_m$ and $x_n(\gamma):=(\gamma_2)_n$.


\begin{claim} \label{Cl.Cu.1} The map $\Gamma\to\cP(\bbN)^2: \gamma \stackrel{\Psi_\lCu}\longmapsto R[\lCu,\gamma]$, 
defined by 
\[
(m,n)\in R[\lCu,\gamma]\text{ if and only if } 
d_m(\gamma)\lCu d_n(\gamma)
\]
(where $\lCu$ is computed in $C^*(\gamma_1)=C^*(\gamma)\otimes \cK$) 
is Borel. 
\end{claim} 

\begin{proof} Recall that a map is Borel if and only if its graph is Borel. 
We have  (writing $d_m$ for $d_m(\gamma)$ and $x_m$ for $x_m(\gamma)$) 
$(m,n)\in R[\lCu,\gamma]$ if and only if 
$(\forall i) (\exists j) \|x_j  d_n x_j^* - d_m\| <1/i$.
Therefore the graph of $\Psi_\lCu$ is equal to 
$\bigcap_i \bigcup_j A_{ij}$ where 
\[
A_{ij}=\{(\gamma,(m,n)): \|x_j(\gamma) d_n(\gamma) x_j(\gamma)^*-d_m(\gamma)\|<1/i\}.
\]
All of these sets are Borel 
since the maps $\gamma\mapsto d_m(\gamma)$ and $\gamma\mapsto d_m(\gamma)$ are, by the above, Borel. 
The map that sends the pair of sequences $(x_j)$ and $(d_j)$ to $R[\lCu,\gamma]$ is therefore Borel. 
The computation of these two sequences from $\gamma$ is Borel by construction, and this completes the proof. 
\end{proof} 

By Claim~\ref{Cl.Cu.1}, for each $\gamma$ we have a preordering $R[\lCu,\gamma]$ on $\bbN$. 
Then  
\[
R[\sCu,\gamma]=\{(m,n): (m,n)\in R[\lCu,\gamma] \ \mathrm{and} \ (n,m)\in R[\lCu,\gamma]\}
\]
is also a Borel function, and it defines a quotient partial ordering on $\bbN$ 
for every $\gamma$.   In what follows we use $[a]$ to denote the Cuntz equivalence class of a positive element of $C^*(\gamma)$.

\begin{claim} \label{Cl.Cu.w}
The map $\Gamma\to\cP(\bbN)^3: \gamma \stackrel{\Psi_+}\longmapsto R[+,\gamma]$, 
defined by 
\[
(m,n,k)\in R[+,\gamma]\text{ if and only if } 
[d_m]+ [d_n]=[d_k] 
\]
(where $+$ is computed in $\Cu(C^*(\gamma))$) 
is Borel. Moreover, it naturally defines a semigroup 
operation on $\bbN/R[\sCu,\gamma]$. 
\end{claim}

\begin{proof} Fix $\gamma$. Let us first prove that the sequence $d_m:=d_m(\gamma)$ 
is such that for all $m$ and $n$ there is $k$ satisfying $[d_m]+[d_n]=[d_k]$. 
Our choice of generating sequence $\gamma_0$ for $\cK$ ensures that each $d_m$ is contained in $C^*(\gamma) \otimes \mathrm{M}_n$ for some $n$, where $\mathrm{M}_1 \subseteq \mathrm{M}_2 \subseteq \mathrm{M_3} \subseteq \cdots$ is a fixed sequence of matrix algebras with union dense in $\cK$.  The $\mathrm{M}_n$ are the bounded operators on $\mathrm{span}(e_1,\ldots,e_n)$.
By construction $\gamma_0$ is closed under finite permutations of the basis $(e_n)$, 
so that for a large enough $l$ the isometry $v_l$ (see above) we have that 
$b_m:= (1 \otimes v_l )d_m (1 \otimes v_l)*$ is both Cuntz equivalent to $d_m$ and orthogonal to $d_n$.  Here the "1" in the first
tensor factor is the unit of $C^*(\gamma)$ if $C^*(\gamma)$ is unital, and the unit of the unitization of $C^*(\gamma)$ otherwise.  Note that $ w_l := b_m(1 \otimes v_l)$ belongs to $C^*(\gamma) \otimes \cK$, and that $w_l d_m w_l^* = b_m^3$ is Cuntz equivalent to and orthogonal to $d_m$.
It follows that 
\[
[d_n + w_l d_m w_l^*] = [d_n] + [w_l d_m w_l^*] = [d_n]+[d_m].
\]
By the defininition of the function $\Tensor$ in \cite[Lemma 3.10]{FaToTo2}, 
for all $l$ we have $v_l d_m v_l^*=d_{r(l)}$ and $d_n+d=d_{k(l)}$ for some $r(l)$ and $k(l)$. 

Now we check that the graph of $\Psi_+$ is Borel. 
This is equivalent to verifying that the graph of the function that maps each triple 
$(\gamma,m,n)$ to the set $X_{\gamma,m,n}$ 
of all $k$ such that $(m,n,k)\in \Psi_+(\gamma)$ is Borel. 
Moreover, a function $\Lambda$ from a Borel space into $\cP(\bbN)$ is Borel if and only if all of the 
sets $\{(\gamma,k): k\in \Lambda(\gamma)\}$ are Borel. 

It will therefore suffice to check that the set $\{(\gamma,(m,n,k)): (m,n,k)\in \Psi_+(\gamma)\}$ 
is Borel. But by the above, $(m,n,k)\in \Psi_+(\gamma)$ is equivalent to 
(writing $d_m$ for $d_m(\gamma)$)
\[
(\exists m) (\forall l\geq m) d_n+ w_l d_m w_l^* \sCu d_k
\]
where $\sCu$ is the Cuntz equivalence relation:  $a \sim b$ iff $a \precsim b$ and $b \precsim a$. 
This is a Borel set, and therefore the map $\Psi_+$ is Borel. 

Clearly, $\Psi_+(\gamma)$ is compatible with $\lCu$ and it defines the addition on $\bbN/R[\sCu,\gamma]$ that coincides with the addition on the Cuntz semigroup.  
\end{proof}

\begin{claim} The map $\Gamma\to\cP(\bbN)^2:\gamma\mapsto R[\ll,\gamma]$, 
defined by 
\[
(m,n)\in R[\ll,\gamma]\text{ if and only if } 
[d_m]\ll [d_n] 
\]
(where $\ll$ is computed in $\Cu(C^*(\gamma))$) 
is Borel. \end{claim} 

\begin{proof} We have $[d_m]\ll [d_n]$ if and only if there exists $j\in \bbN$ such that 
 $d_m\lCu (d_n-1/j)_+$ (\cite{apt}). 
 Recalling that $d_{\pi(n,j)}=(d_n-1/j)_+$ for all $n$ and $j$, 
 we see that is equivalent to 
 \[
 (\exists j)(m,\pi(n,j))\in R[\lCu,\gamma]
 \]
 and therefore the map is Borel. 
\end{proof} 

Collecting these three claims we see that the map which sends $\gamma$ to an element of $\BCu$ representing 
$\Cu(C^*(\gamma))$---call it $\Phi$---is Borel. 
\end{proof} 


\subsection{The radius of comparison}

The radius of comparison is a notion of dimension for noncommutative spaces which is useful for distinguishing simple nuclear \CSTAR algebras and is connected deeply to Elliott's classification program (see \cite{tcmp} and \cite{et}).

Consider the standard space $\BCu_u = \BCu \times \mathbb{N}$, where the second co-ordinate of $(D,e) \in \BCu_u$ represents a distinguished element of $D$.  Let $\Cu_u$ denote the category of Cuntz semigroups with a distinguished compact element.  It is straightforward, by following the proof of Proposition \ref{P.Cu.2}, to verify that there is a Borel map $\Psi\colon \Gamma_u \to \BCu_u$ such  $\Psi(\gamma) = (D,[1_{C^*(\gamma)}])$, where $D$ is (identified with) a countable sup-dense subsemigroup of $\Cu(C^*(\gamma))$.



If $(S,e) \in \Cu_u$, then the {\it radius of comparison of $S$} (relative to $e$), denoted by $r(S,e)$, is defined by
\[
r(S,e) = \inf \{ m/n \ | \ m, n \in \mathbb{N} \ \wedge \ x \leq y \mathrm{\ in \ } S \mathrm{\ whenever \ } (n+1)x + me \leq ny \}
\]
if this infimum exists, and by $r(S,e)= \infty$ otherwise.
Of course, this definition makes sense for any ordered semigroup with a distinguished element $e$, e.g., an element $(D,e)$ of $\BCu \times \mathbb{N}$, so we can equally well define $r(D,e)$ in the same way.

\begin{prop}\label{rccountable}  Let $(S,e) \in \Cu_u$, and let $D \subseteq S$ be a countable sup-dense subsemigroup of $S$ containing $e$.  It follows that, with respect to the common element $e$, 
$r(S,e)=r(D,e)$.
\end{prop} 

\begin{proof}  We suppress the $e$ and write only $r(D)$ and $r(S)$.
It is clear that $r(D) \leq r(S)$.  Given $\epsilon>0$, we will prove $r(S) \leq r(D) + \epsilon$.  Choose $m,n \in \mathbb{N}$ to satisfy
\[
r(D) < m/n < r(D) + \epsilon.
\]
Let $x,y \in S$ satisfy 
\[
(n+1)x + me \leq ny.
\]
There are rapidly increasing sequences $(x_k)$ and $(y_k)$ in $D$ having suprema $x$ and $y$, respectively.  Since $e$ is compact, so is $me$, i.e, $me \ll me$.  Since $(n+1)x_k \ll (n+1)x$ for any $k$, we can use the fact that addition respects $\ll$ to conclude that 
\[
(n+1)x_k + me \ll (n+1)x + me \leq ny.
\]
It follows that 
\[
(n+1)x_k + me \ll ny.
\]
Since the operation of addition respects the operation of taking suprema, we have $\sup ny_l = ny$, whence for some (and hence all larger) $l_k \in \mathbb{N}$ we have
\[
(n+1)x_k +me \leq ny_{l_k}.
\]
Now since $m/n > r(D)$ we conclude that $x_k \leq y_{l_k}$.  Taking suprema yields $x \leq y$, proving that $r(S) \leq m/n < r(D)+ \epsilon$, as desired.
\end{proof}

\begin{prop}\label{rcborel}
The map $\mathrm{rc}:\Gamma_u \to \mathbb{R}^+ \cup \{\infty\}$ given by $\mathrm{rc}(\gamma) = r(\Cu(C^*(\gamma)), [1_{C^*(\gamma)}])$ is Borel.
\end{prop}

\begin{proof}
The map $\Psi: \Gamma_u \to \BCu_u$ is Borel and satisfies $r(\Psi(\gamma)) = r(\Cu(C^*(\gamma)), [1_{C^*(\gamma)}])$ by Proposition \ref{rccountable}.  It will therefore suffice to prove that $r: \BCu_u \to \mathbb{R} \cup \{\infty\}$ is Borel.  For $m,n \in \mathbb{N}$ the set
\[
A_{m,n} = \{ (D,e) \in \BCu_u \ | \ (\forall x,y \in D) (n+1)x + me \leq ny \Rightarrow x \leq y \}
\]
is Borel.
Define a map $\zeta_{m,n}: \BCu_u \to \mathbb{R}^+ \cup \{\infty\}$ by declaring that $\zeta_{m,n}(D,e) = m/n$ if $(D,e) \in A_{m,n}$ and $\zeta_{m,n}(D,e) = \infty$ otherwise.  Viewing the $\zeta_{m,n}$ as co-ordinates we get a Borel map $\zeta:\BCu_u \to (\mathbb{R}^+ \cup \{\infty\})^{\mathbb{N}^2}$ in the obvious way, and $r(D,e) = \inf \zeta(D,e)$.  This shows that $r$ is Borel, as desired.
\end{proof}


\section{Other Invariants}\label{S.other}

\subsection{Theory of a \CSTAR algebra }

Unlike other invariants of \CSTAR algebras treated in this paper, the
\emph{theory} $\Th(A)$ of a \CSTAR algebra  $A$ comes from logic. By the 
metric version of the Keisler--Shelah theorem (\cite[Theorem 5.7]{BYBHU}), it has the property 
that two \CSTAR algebras have isomorphic ultrapowers if and only if they have the same theory.
It should be emphasized that the ultrapowers may have to be associated with ultrafilters on 
uncountable sets, even if the algebras in question are separable. 
A comprehensive treatment of model theory of bounded metric structures is given 
in \cite{BYBHU}, and  
model-theoretic study of \CSTAR algebras and tracial von Neumann algebras was initiated in 
\cite{FaHaSh:Model1} and \cite{FaHaSh:Model2}. We refer the reader to these papers
for more details, background, and applications.

We now give a special case of the definition of a formula
(\cite[\S 2.4]{FaHaSh:Model2}) in the case of \CSTAR algebras (cf. \cite[\S 2.3.1]{FaHaSh:Model2} 
and \cite[\S 3.1]{FaHaSh:Model2}). 
A \emph{term} is a *-polynomial. 
A \emph{basic formula} is an expression of the form $\|P(x_0,\dots, x_{n-1})\|$ where
$P(x_0,  \dots, x_{n-1})$ is a term in variables $x_0,\dots, x_{n-1}$. 
\emph{Formulas} are elements of the smallest set $\bbF$  that contains all basic formulas and
has the following closure properties (we suppress free variables in order to increase readability).  
\begin{enumerate}
\item [(F1)]  If $f:\bbR^n \rightarrow \bbR$ is continuous and
$\varphi_1,\ldots,\varphi_n$ are formulas, then
$f(\varphi_1,\ldots,\varphi_n)$ is a formula.
\item [(F2)] If $\varphi$ is a formula,  $K\geq \bbN$ is a natural number, and 
$x$ is  a variable  
then both $\sup_{\|x\|\leq K} \varphi$ and $\inf_{\|x\|\leq K} \varphi$ are formulas.
\end{enumerate}
Equivalently, formulas are obtained from basic formulas by finite application of the above two operations. 

The quantifiers in  this logic  are  $\sup_{\|x\|\leq 1}$ and $\inf_{\|x\|\leq 1}$. 
A variable appearing in a formula $\psi$ outside of the  scope of its  quantifiers
 (i.e., any $\varphi$ as in (F2)) 
is \emph{free}. 

As customary in logic we list all free variables occurring in 
a fornula $\varphi$ and write $\varphi(x_0,\dots, x_{n-1})$. 
A formula $\varphi(x_0,\dots, x_{n-1})$ 
is interpreted  in  a \CSTAR algebra  $A$ in a natural way. Given $a_0,\dots, a_{n-1}$ in $A$, 
one defines the value $\varphi(a_0,\dots, a_{n-1})^A$ recursively on the complexity 
of formula $\varphi$. As $a_0,\dots, a_{n-1}$ vary, one obtains  a  function
from $A^n$ into $\bbR$ whose restriction to any bounded ball of $A$ is uniformly continuous
(\cite[Lemma~2.2]{FaHaSh:Model2}). 
A \emph{sentence} is a formula with no free variables. 
If $\varphi$ is a sentence then the interpretation 
$\varphi^A$ is a constant function and we identify it with the corresponding real number. 
\emph{Theory} of a \CSTAR algebra $A$ is the map $\varphi\mapsto \varphi^A$ 
from  the set of all sentences  into $\bbR$.

The above definition results in an uncountable set of formulas. However, 
by restricting terms to *-polynomials  with complex rational coefficients and
continuous functions $f$ in (F1) to polynomials with rational coefficients, 
one obtains a countable set of formulas that approximate every other formula arbitrarily well. 
Let $\bbS_0$ denote the set of  all sentences  in this countable set. 
Clearly, the restriction of $\Th(A)$ to $\bbS_0$ determines $\Th(A)$
and we can therefore consider a closed subset of $\bbR^{\bbS_0}$ 
to be a Borel space of all theories of \CSTAR algebras. 

\begin{prop} The function from $\hat \Gamma$ into $\bbR^{\bbS_0}$ that
associates $\Th(C^*(\gamma))$ to $\gamma\in \hat\Gamma$ is Borel. 
\end{prop}

\begin{proof} This is an immediate consequence of the lemma given 
below. 
\end{proof} 

\begin{lemma} Given a formula $\varphi(x_0,\dots, x_{n-1})$, the 
map that associates $\varphi(\gamma_{k(0)}, \dots, \gamma_{k(n-1)})^{C^*(\gamma)}$ 
to a pair $(\gamma, \vec k)\in\hat\Gamma\times \bbN^n$  is Borel. 
\end{lemma} 

\begin{proof} By recursion on the complexity of $\varphi$. We suppress parameters $x_0,\dots, x_{n-1}$ for simplicity.  
If $\varphi$ is basic, then the lemma reduces to the fact that evaluation of
the norm of a *-polynomial is Borel-measurable. 
The case when  $\varphi$ is of the form $f(\varphi_0,\dots, \varphi_{n-1})$ as in (F1) 
and lemma is true for each $\varphi_i$ is trivial. 

Now assume $\varphi$ is of the form $\sup_{\|y\|\leq K} \psi(y)$ with $K\geq 1$.  
Function $t_K\colon \bbR\to \bbR$  defined by $t(r)=r$, if $r\leq K$ and $t(r)=1/r$ if $r>K$
is continuous, and since  
\[
\varphi^{C^*(\gamma)}=\sup_{i\in \bbN} \psi(_K(\|\gamma_i\|)\gamma_i)
\]
we conclude that the computation of $\varphi$ is Borel as a supremum of countably many Borel functions. 
The case when $\varphi$ is $\inf_{\|y\|\leq K} \psi(y)$ is similar. 
\end{proof} 

We note that an analogous proof shows  that the computation of a theory of a 
  tracial von Neumann algebra is  a Borel function from the corresponding subspace of 
   Effros--Mar\'echal space into $\bbR^{\bbS_0}$. 
  
\subsection{Stable and real rank}\label{S.ranks}

The stable rank $\mathrm{sr}(A)$ of a unital \CSTAR algebra $A$ is the least natural number $n$ such that 
\[
Lg_n = \left\{ (a_1,\ldots,a_n) \in A^n \ | \ \exists \ b_1,\ldots, b_n \in A \ \mathrm{such \ that} \left\| \sum_{i=1}^n b_ia_i - \mathbf{1}_A \right\| < 1 \right\}
\]
is dense in $A^n$, if such exists, and $\infty$ otherwise.  The real rank $\mathrm{rr}(A)$ is the least natural number $n$ such that 
\[
Lg_{n+1}^{sa} = \left\{ (a_1,\ldots,a_{n+1}) \in A_{sa}^{n+1} \ | \ \exists \ b_1,\ldots, b_{n+1} \in A_{sa} \ \mathrm{such \ that} \left\| \sum_{i=1}^{n+1} b_ia_i - \mathbf{1}_A \right\| < 1 \right\}
\]
where $A_{sa}$ denotes the self-adjoint elements of $A$.  Again, if no such $n$ exists, we say that $\mathrm{rr}(A)=\infty$.

\begin{theorem}
The maps $\mathrm{SR}: \Gamma \to \mathbb{N} \cup \{\infty\}$ and $\mathrm{RR}: \Gamma \to \mathbb{N} \cup \{\infty\}$ given by 
$\mathrm{SR}(\gamma) = \mathrm{sr}(C^*(\gamma))$ and $\mathrm{RR}(\gamma) = \mathrm{rr}(C^*(\gamma))$, respectively, are Borel.
\end{theorem}

\begin{proof}
We treat only the case of $\mathrm{SR}(\bullet)$;  the case of $\mathrm{RR}(\bullet)$ is similar.  We have
\[
C^*(\gamma) \in Lg_n \Leftrightarrow (\forall i_1 < i_2 < \cdots i_n)(\exists j_1 < j_2 <\cdots < j_n): \left\| \sum_{k=1}^n \gamma_{j_k}\gamma_{i_k} - \mathbf{1}_A \right\| < 1.
\]
For fixed $i_1 < i_2 < \cdots i_n$ and $j_1 < j_2 <\cdots < j_n$, the set on the left hand side is norm open in all co-ordinates
of $B(\mathcal{H})^\mathbb{N} =\Gamma$, and hence Borel.  The theorem follows immediately.
\end{proof}

\appendix
\section{Appendix (with Caleb Eckhardt)}

\subsection{$\mathcal{Z}$-stability}

The Jiang-Su algebra $\mathcal{Z}$ plays a central role in the classification theory of nuclear separable \CSTAR algebras.  Briefly, one can expect good classification results for algebras which are $\mathcal{Z}$-stable, i.e., which satisfy $A \otimes \mathcal{Z} \cong A$ (see \cite{et} for a full discussion).  We prove here that the subset of $\Gamma$ consisting of $\mathcal{Z}$-stable algebras is Borel.  

It was shown in \cite{js} that $\mathcal{Z}$ can be written as the limit of a \CSTAR algebra inductive sequence
\[
Z_{n_1,n_1+1} \stackrel{\phi_1}{\longrightarrow}
Z_{n_2,n_2+1} \stackrel{\phi_2}{\longrightarrow}
Z_{n_3,n_3+1} \stackrel{\phi_3}{\longrightarrow} \cdots
\]
where
\[
Z_{n,n+1} = \{f \in C([0,1]; \mathrm{M}_n \otimes \mathrm{M}_{n+1}) \ | \  f(0) \in \mathrm{M}_n \otimes \mathbf{1}_{n+1}, \ f(1) \in \mathbf{1}_n \otimes \mathrm{M}_{n+1} \}
\]
is the {\it prime dimension drop algebra} associated to $n$ and $n+1$.  The property of being $\mathcal{Z}$-stable for a \CSTAR algebra $A$ can be characterized as the existence, for each $n$, of a sequence of $*$-homomorphisms $\psi_k: Z_{n,n+1} \to A$ with the property that 
\[
\| [\psi_k(f), x]\| \to 0, \ \forall f \in Z_{n,n+1}, \ \forall x \in A.
\]
The algebra $Z_{n,n+1}$ was shown in \cite{js} to admit weakly stable relations, i.e., there exists a finite set of relations $\mathcal{R}_n$ in $l(n)$ indeterminates with the following properties:  
\begin{enumerate}
\item[(i)] the universal \CSTAR algebra for $\mathcal{R}_n$ is $\mathcal{Z}_{n,n+1}$;
\item[(ii)] for every $\epsilon>0$ there exists $\delta(\epsilon)>0$ such that if $g_1,\ldots,g_{l(n)}$ are elements in a \CSTAR algebra $A$ which satisfy the relations $\mathcal{R}_n$ to within $\delta(\epsilon)$, then there exist $h_1,\ldots,h_{l(n)} \in A$ which satisfy the relations $\mathcal{R}_n$ precisely and for which $\|g_i-h_i\| < \epsilon$.
\end{enumerate}
What's really relevant for us is that if $g_1,\ldots,g_{l(n)}$ are elements in a \CSTAR algebra $A$ which satisfy the relations $\mathcal{R}_n$ to within $\delta(\epsilon)$, then there is a $*$-homomorphism $\eta:Z_{n,n+1} \to A$ such that the indeterminates for $\mathcal{R}_n$ are sent to elements $\epsilon$-close to $g_1,\ldots,g_{l(n)}$, respectively.

Using the equivalence of the parameterizations $\Gamma$ and $\hat\Gamma$ for separable \CSTAR algebras, we may assume that the sequence $\gamma$ in $B(\mathcal{H})^\mathbb{N}$ giving rise to C$^*(\gamma)$ is in fact dense in C$^*(\gamma)$.  The $\mathcal{Z}$-stability of C$^*(\gamma)$ for $\gamma = (a_i)_{i \in \mathbb{N}}$ is then equivalent to the following statement:
\vspace{2mm}
\begin{quote}
$(\forall k)(\forall n)(\forall j)(\exists (i_1,\ldots i_{l(n)}))$ such that $a_{i_1},\ldots,a_{i_{l(n)}}$ are a $\delta(1/k)$-representation of $\mathcal{R}_n$ and $\| [a_{i_s},a_m] \| < 1/k$ for each $s \in \{1,\ldots,l(n)\}$ and $m \in \{1,\ldots, j\}$.
\end{quote}
\vspace{2mm}
\noindent
If we fix $k,n,j$ and $(i_1,\ldots,i_{l(n)})$ it is clear that those $\gamma \in \hat\Gamma$ for which $(a_{i_1},\ldots,a_{i_{l(n)}})$ satisfy the latter two conditions above form a norm open and hence Borel set.  This theorem follows immediately:
\begin{theorem}\label{zstable}
$\{\gamma \in \Gamma \ | \ C^*(\gamma) \ \mathrm{is \ } \mathcal{Z}\mathrm{-stable} \}$
is Borel.
\end{theorem}

\subsection{Nuclear dimension}

A completely positive map $\phi\colon A\to B$ between \CSTAR algebras has \emph{order zero} if 
it is orthogonality preserving, in the sense that for positive $a,b$ in $A$
we have $ab=0$ implies $\phi(a)\phi(b)=0$. 
 
A \CSTAR algebra $A$ has \emph{nuclear dimension} at most $n$ if the following holds. 
For every $\e>0$, for every finite $F\subseteq A$, there are finite-dimensional 
\CSTAR algebras $B_1,\dots, B_n$ 
and completely positive maps $\psi\colon A\to \bigoplus_{i=1}^n B_i$ and 
$\phi\colon \bigoplus_{i=0}^n B_i\to A$ such that 
\begin{enumerate}
\item $\|\psi\circ \phi(a)-a\|<\e$ for all $a\in F$, 
\item $\|\psi\|\leq 1$, and 
\item $\phi\rs B_i$ has order zero for every $i\leq n$. 
\end{enumerate}
The \emph{nuclear dimension} of $A$, denoted $\dimnuc(A)$, 
 is the minimal $n$ (possibly $\infty$) 
such that $A$ has nuclear dimension $\leq n$ (see \cite{WiZa:Nuclear}). 

The proof of the following theorem is based on Effros's proof that nuclear \CSTAR algebras form a Borel subset of~$\Gamma$ (see \cite[\S 5]{Kec:C*}). 
 
\begin{theorem} 
The  map $\dimnuc\colon \Gamma\to \bbN\cup \{\infty\}$ is Borel. 
\end{theorem}

\begin{proof} It suffices  to check that the set of all $\gamma$ 
such that $\dimnuc(C^*(\gamma))\leq n$ is Borel.  
Let $M_n(A^*)$ denote the space of  $n\times n$ matrices of the elements
of the Banach space dual of $A$, naturally identified with the space 
of bounded linear maps from $A$ into $M_n(\bbC)$. We consider this space with respect to the 
weak*-topology, which makes it into a $K_\sigma$ Polish space.
   
As demonstrated in \cite[\S 5]{Kec:C*}, there is a Borel map 
$\Upsilon\colon \Gamma\times n\to (M_n(A^*))^{\bbN}$
such that $\Upsilon(\gamma,n)$ enumerates a dense subset of 
the (weak*-compact) set of completely positive maps  from $A$ into $M_n(\bbC)$. 
Note that order zero maps form a closed subset of the set of completely positive maps, 
and the proof from \cite{Kec:C*} provides a Borel enumeration of a countable dense set of
completely positive order zero maps. 

Again as in \cite[\S 5]{Kec:C*}, we use the fact that a map $\psi$ from $M_n(\bbC)$ to $A$ 
is completely positive  if and only if $\psi((x_{ij})=\sum_{i,j} x_{i,j} a_{i,j}$ where $(a_{i,j})$
is a positive element of $M_n(A)$ of norm $\leq1$. 
By \cite[Lemma 3.10]{FaToTo2} and \cite[Lemma 3.13]{FaToTo2} there is 
a Borel function $\Xi\colon \Gamma\to (\Gamma^{n\times n})^{\bbN}$
such that $\Xi(\gamma)$ is an enumeration of a countable dense set
of such $(a_{i,j})$. 

Inspection of (1)--(3) in 
the definition of $\dimnuc(C^*(\gamma))\leq n$ reveals that the verification of these conditions
is only required over countable subsets of the allowable $\phi$, $\psi$, and $a$, subsets which
are computed in a Borel manner from $\gamma$ using the maps $\Upsilon$, $\Xi$, and $\gamma$ itself,
respectively.  It follows that the set of $\gamma$ for which $\dimnuc(C^*(\gamma))\leq n$ is Borel.
\end{proof}

\bigskip

{\it Acknowledgement.} C. Eckhardt was supported by NSF grant DMS-1101144.  I. Farah was partially supported by NSERC. A. Toms was supported by NSF grant DMS-0969246 and the 2011 AMS Centennial Fellowship. A. T\"ornquist was supported by a Sapere Aude fellowship (level 2) from DenmarkÕs Natural Sciences Research Council, no. 10-082689/FNU, and a Marie Curie re-integration grant, no. IRG-249167, from the European Union.

 \bibliography{invariants}
\bibliographystyle{plainnat}

\end{document}